\documentclass[10pt]{amsart}

\usepackage{hyperref}
\newtheorem{thm}{Theorem}[section]
\newtheorem{prop}[thm]{Proposition}

\newtheorem{lem}[thm]{Lemma}

\theoremstyle{definition}

\numberwithin{equation}{section}
\newtheorem{rem}[thm]{Remark}

\newtheorem{defn}[thm]{Definition}

\def\bbP{\mathbb{P}}
\def\bbQ{\mathbb{Q}}
\def\bbT{\mathbb{T}}

\def\bbZ{\mathbb{Z}}

\def\bfb{\mathbf{b}}
\def\bfd{\mathbf{d}}
\def\bfe{\mathbf{e}}
\def\bfp{\mathbf{p}}

\def\bfx{\mathbf{x}}
\def\bfy{\mathbf{y}}
\def\bfz{\mathbf{z}}
\def\calA{\mathcal{A}}
\def\calF{\mathcal{F}}

\def\ve{\varepsilon}

\allowdisplaybreaks[1]

\begin{document}
\bibliographystyle{amsalpha}

\title[Structure of seeds in generalized cluster algebras]
{Structure of seeds in  generalized cluster algebras}

\author{Tomoki Nakanishi}
\address{\noindent Graduate School of Mathematics, Nagoya University, 
Chikusa-ku, Nagoya,
464-8604, Japan}
\email{nakanisi@math.nagoya-u.ac.jp}


\date{}
\maketitle

\begin{abstract}
We study generalized cluster algebras introduced by
Chekhov and Shapiro.
When the coefficients satisfy the normalization and quasi-reciprocity conditions,
one can naturally extend the structure theory of seeds in
the ordinary cluster algebras by Fomin and Zelevinsky
to generalized cluster algebras.
As the main result,
we obtain formulas expressing cluster variables and coefficients in terms of
$c$-vectors, $g$-vectors, and $F$-polynomials.
\end{abstract}

\section{Introduction}
In \cite{Chekhov11} Chekhov and Shapiro introduced
{\em generalized cluster algebras},
which naturally generalize the ordinary cluster algebras by
Fomin and Zelevinsky \cite{Fomin02}.
In generalized cluster algebras, the celebrated {\em binomial\/} exchange relation
for cluster variables of ordinary cluster algebras
\begin{align}
\begin{split}
x'_k x_k &=
p_k^- \prod_{j=1}^n x_j^{[-b_{jk}]_+ }
+
p_k^+ \prod_{j=1}^n  x_j^{[b_{jk}]_+ }
\\
&=
\Biggl(\prod_{j=1}^n  x_j^{[-b_{jk}]_+ }\Biggr)
(p_k^{-} + p_k^+ w_k),
\quad
w_k=\prod_{j=1}^n  x_j^{b_{jk}}
\end{split}
\end{align}
is replaced by the {\em polynomial\/} one of arbitrary degree $d_k\geq 1$,
\begin{align}
\label{eq:rel1}
x'_k x_k =
\Biggl(\prod_{j=1}^n  x_j^{[-\beta_{jk}]_+ }\Biggr)^{d_k}
\sum_{s=0}^{d_k}
 p_{k,s} w_k^s,
\quad
w_k=\prod_{j=1}^n  x_j^{\beta_{jk}},
\end{align}
where $\beta_{jk}=b_{jk}/d_k$  are assumed to be integers
and the coefficients $p_{k,s}$ should be also mutated appropriately.
This generalization is expected to be natural, since it originates in  the transformations preserving the associated Poisson bracket
 \cite{Gekhtman05}.
 In fact, it was shown in \cite{Chekhov11} that 
 the generalized cluster algebras have the {\em Laurent property},
 which is regarded as the most characteristic feature of the ordinary cluster algebras.
 Furthermore, it was also shown   in \cite{Chekhov11} that  the {\em finite-type classification\/} of the generalized cluster algebras reduces to the one for the ordinary case.
 These results already imply that, despite the apparent complexity of their exchange relations
 \eqref{eq:rel1}, generalized cluster algebras may be well-controlled like the ordinary ones.
 See also \cite{Rupel13} for the result on greedy bases in rank 2 generalized cluster algebras.
 
 Besides  the above cluster-algebra-theoretic interest,
 the generalized cluster algebra structure  naturally appears for the Teichm\"uller spaces of
 Riemann surfaces with orbifold points \cite{Chekhov11}.
 More recently, it also appears in representation theory of quantum affine algebras
 \cite{Gleitz14}
 and also in the study of WKB analysis \cite{Iwaki14b}.
In view of these developments, and also for  potentially more versatility of polynomial
exchange relations than the binomial one, it is not only natural but also necessary to develop
a structure theory of seeds in generalized cluster algebras
which is parallel to the one for the ordinary cluster algebras by \cite{Fomin07}.
The core notion of the  theory of \cite{Fomin07} is 
a cluster pattern with {\em principal coefficients},
from which
 other important notions
such as {\em $c$-vectors}, {\em $g$-vectors}, and {\em $F$-polynomials\/} are also induced.
Then, the main result of \cite{Fomin07} is the formulas expressing
cluster variables and coefficients in terms of $c$-vectors,
$g$-vectors, and $F$-polynomials.
These formulas are  especially important in view of the categorification
of  cluster algebras by (generalized) cluster categories (see
\cite{Plamondon10b} and reference therein).

The purpose of this paper is
to provide  results parallel to the above ones
for generalized cluster algebras.
To be more precise,
we consider a class of generalized cluster algebras
whose coefficients
satisfy the {\em normalization condition\/} and 
what we call the {\em quasi-reciprocity condition}.
For this class of generalized cluster algebras,
we introduce the notions of
a cluster pattern with principal coefficients,
$c$-vectors, $g$-vectors, and $F$-polynomials.
Then, as a main result, we obtain
the formulas expressing cluster variables and coefficients in terms of
$c$-vectors, $g$-vectors, and $F$-polynomials,
which are parallel to the ones in \cite{Fomin07}.
To summarize, 
{\em generalized cluster algebras preserve
essentially every feature of the ordinary ones},
and this is the main message of the paper.

\bigskip
{\em Acknowledgements.}
We thank Anne-Sophie Gleitz, Kohei Iwaki, and Michael Shapiro
for useful discussions and communications.

\section{Generalized cluster algebras}

In this section we recall basic notions of generalized cluster algebras
following \cite{Chekhov11}.
However, we slightly modify the setting of  \cite{Chekhov11}
to match   the setting of (ordinary) cluster algebras  in \cite{Fomin07}.

\subsection{Generalized seed mutations}
Throughout the paper we always assume that any matrix is an {\em integer\/} matrix.

Recall that a matrix $B=(b_{ij})_{i.j=1}^n$ is said to be {\em skew-symmetrizable\/}
if there is an $n$-tuple of positive integers $\bfd=(d_1,\dots,d_n)$
such that $d_ib_{ij}=-d_jb_{ji}$.

We start by fixing a semifield $\bbP$, whose addition is denoted by $\oplus$.
Let $\bbZ\bbP$ be the group ring of $\bbP$,
and let $\bbQ\bbP$ be the field of the fractions of $\bbZ\bbP$.
Let $w_1,\dots,w_n$ be any algebraic independent variables,
and let $\calF=\bbQ\bbP(w)$ be the field of the rational functions
in $w=(w_1$,\dots,
$w_n)$ with coefficients in $\bbQ\bbP$.

The following definition is the usual one \cite{Fomin07}.

\begin{defn} A (labeled) seed in $\bbP$ is a triplet $(\bfx,\bfy,B)$
such that
\begin{itemize}
\item $B$ is a skew-symmetrizable matrix, called an {\em exchange matrix},
\item $\bfx=(x_1,\dots,x_n)$ is an $n$-tuple of
elements in $\calF$, called  {\em cluster variables\/} or {\em $x$-variables},
\item $\bfy=(y_1,\dots,y_n)$ is an $n$-tuple of
elements in $\bbP$,
 called  {\em coefficients\/} or {\em $y$-variables}.
\end{itemize}
\end{defn}

Next we introduce a pair $(\bfd,\bfz)$ of  data  for generalized seed mutations
we consider.
Firstly, $\bfd=(d_1,\dots,d_n)$ is an $n$-tuple of positive integers,
and we call these integers the {\em mutation degrees}.
We stress that we do {\em not\/} impose the skew-symmetric condition
$d_ib_{ij}=-d_jb_{ji}$.
Secondly, $\bfz$ is a family of elements in $\bbP$,
\begin{align}
\bfz=(z_{i,s})_{ i=1,\dots,n; s=1,\dots,d_i-1}
\end{align}
satisfying the following condition:
\par
(reciprocity)
\begin{align}
\label{eq:p1}
z_{i,s}=z_{i,d_i-s}
\quad (s=1,\dots,d_i-1).
\end{align}
We call them 
 the {\em frozen coefficients},
since they are not ``mutated",
or simply the {\em $z$-variables}.
We also set
\begin{align}
z_{i,0}=z_{i,d_i}=1.
\end{align}

For $\bfd=(1,\dots,1)$, $\bfz$ is empty,
and  it reduces to the ordinary case. 
(Here and below, ``ordinary" means the case of ordinary cluster algebras).

\begin{defn}
\label{defn:mut1} Let $(\bfd,\bfz)$ be given as above.
For any seed $(\bfx,\bfy,B)$ in $\bbP$ and $k=1,\dots,n$,
 the {\em   $(\bfd,\bfz)$-mutation of $(\bfx,\bfy,B)$ at $k$\/}
 is another seed $(\bfx',\bfy',B')=\mu_k(\bfx,\bfy,B)$ in $\bbP$
 defined by the following rule:
 \begin{align}
 \label{eq:bmut1}
b'_{ij}&=
\begin{cases}
-b_{ij}& \mbox{$i=k$ or $j=k$}\\
b_{ij}+d_k
\left([- b_{ik}]_+ b_{kj} + b_{ik}[ b_{kj}]_+\right)
& \mbox{$i,j\neq k$,}\\
\end{cases}
\\
\label{eq:ymut1}
{y}'_i&=
\begin{cases}
\displaystyle
{y}_k^{-1}
 & i=k\\
 \displaystyle
{y}_i
\Biggl(
{y}_k^{[\ve {b}_{ki}]_+} 
\Biggr)^{d_k}
\Biggl(
\bigoplus_{s=0}^{d_k}
z_{k,s} {y}_k^{\ve s}
\Biggr)^{-{b}_{ki}}
& i\neq k,\\
\end{cases}
\\
\label{eq:xmut1}
x'_i&=
\begin{cases}
\displaystyle
x_k^{-1}
\Biggl(\prod_{j=1}^n 
x_j^{[-\ve {b}_{jk}]_+}
\Biggr)^{d_k}
\frac{\displaystyle
\sum_{s=0}^{d_k}
z_{k,s} \hat{y}_k^{\ve s}
}
{\displaystyle
\bigoplus_{s=0}^{d_k}
z_{k,s} {y}_k^{\ve s}
}
 & i=k\\
x_i& i\neq k,\\
\end{cases}
\end{align}
where $\ve=\pm 1$, $[a]_+=\max(a,0)$, and we set
\begin{align}
\label{eq:yhat1}
\hat{y}_i
&= y_i \prod_{j=1}^n
x_j^{{b}_{ji}}.
\end{align}
 When the data $(\bfd,\bfz)$ is clearly assumed,
 we may drop the prefix and simply call it the {\em (generalized) mutation}.

\end{defn}

 Let $D=(d_i\delta_{ij})_{i,j=1}^n$ be the diagonal matrix with
diagonal entries  $\bfd$.
It is important to note that
the mutation \eqref{eq:bmut1} is equivalent to the {\em ordinary\/} mutation
of exchange matrices between $DB$ and $DB'$, and also between $BD$ and $B'D$
in \cite{Fomin07}.

The following properties are easy to confirm:
\begin{itemize}
\item The formulas \eqref{eq:ymut1} and \eqref{eq:xmut1} are
{\em independent\/} of the choice of the sign $\ve$ due to
\eqref{eq:p1}.
\item The mutation $\mu_k$ is involutive,
i.e., $\mu_k(\mu_k(\bfx,\bfy,B))=(\bfx,\bfy,B)$.
\end{itemize}

\begin{rem}
\label{rem:CS1}
Here we transposed every matrix in  \cite{Chekhov11}.
Also the matrix $B$ therein is the matrix $DB^{T}$ here,
and  $\beta_{ij}$ therein is $b_{ji}$ here.
\end{rem}

\begin{rem}
In this paper we do not use the freedom of the choice of sign $\ve$
in \eqref{eq:ymut1} and \eqref{eq:xmut1},
and it can be  safely set as $\ve=1$ throughout.
Nevertheless, we keep it in all formulas involved since
it is useful for several purposes, for example,
to consider {\em signed mutations}
which appeared in \cite{Iwaki14b}.
\end{rem}
\begin{prop} Under the mutation $\mu_k$, the $\hat{y}$-variables \eqref{eq:yhat1} mutate 
in the same way as the $y$-variables, namely,
\begin{align}
\label{eq:yhatmut1}
\hat{y}'_i&=
\begin{cases}
\displaystyle
\hat{y}_k^{-1}
 & i=k\\
 \displaystyle
\hat{y}_i
\Biggl(
\hat{y}_k^{[\ve {b}_{ki}]_+} 
\Biggl)^{d_k}
\Biggl(
\sum_{s=0}^{d_k}
z_{k,s} \hat{y}_k^{\ve s}
\Biggr)^{-{b}_{ki}}
& i\neq k.\\
\end{cases}
\end{align}
\end{prop}
\begin{proof}
This is proved using the  technique in \cite[Proposition 3.9]{Fomin07}.
\end{proof}

 Next let us explain how our setting is regarded as a specialization of
the setting of \cite{Chekhov11}.
In \cite{Chekhov11} a seed in $\bbP$ is defined as a triplet
$(\bfx,\bfp,B)$, where $\bfx$ and $B$ are the same ones here
(up to the identification of $B$ as in Remark \ref{rem:CS1}),
but $\bfp$ is a family of elements in $\bbP$,
\begin{align}
\label{eq:p1a}
\bfp=(p_{i,s})_{i=1,\dots,n; s=0,\dots,d_i}.
\end{align}
Then, for the mutation
$(\bfx',\bfp',B')=\mu_k(\bfx,\bfp,B)$,
 the following  formulas replace \eqref{eq:ymut1}
and \eqref{eq:xmut1}:
 \begin{align}
\label{eq:pmut1}
\begin{split}
p'_{k,s}&=p_{k,d_k-s},\\
\frac{p'_{i,s}}{p'_{i,0}}&=
\begin{cases}
 \displaystyle
\frac{p_{i,s}}{p_{i,0}}
(p_{k,d_k}^{ {b}_{ki}} )^s
& i\neq k,\ b_{ki}\geq 0\\
 \displaystyle
\frac{p_{i,s}}{p_{i,0}}
(p_{k,0}^{ {b}_{ki}} )^s
& i\neq k,\ b_{ki}\leq 0,\\
\end{cases}
\end{split}
\\
\label{eq:xmut2}
x'_i&=
\begin{cases}
\displaystyle
x_k^{-1}
\Biggl(\prod_{j=1}^n 
x_j^{[- {b}_{jk}]_+}
\Biggr)^{d_k}
\Biggl(\sum_{s=0}^{d_k}
p_{k,s} {u}_k^{ s}
\Biggr)
&
 i=k\\
x_i& i\neq k,\\
\end{cases}
\end{align}
where
\begin{align}
\label{eq:z1}
u_i
&= \prod_{j=1}^n
x_j^{{b}_{ji}}.
\end{align}
Now, let us start from  a seed $(\bfx,\bfy,B)$ in our setting.
Comparing \eqref{eq:xmut1} and \eqref{eq:xmut2}, we  naturally identify
\begin{align}
\label{eq:pu1}
p_{i,s}=\frac{z_{i,s} {y}_i^{ s}}{\bigoplus_{r=0}^{d_i}
z_{i,r} {y}_i^{ r}}.
\end{align}
Then, it is easy to check that  the mutation \eqref{eq:pmut1} 
follows from \eqref{eq:p1} and \eqref{eq:ymut1}.
Moreover, the  specialization \eqref{eq:pu1} satisfies the following properties:

 (normalization)
 \begin{align}
\label{eq:p2}
\bigoplus_{s=0}^{d_i} p_{i,s} &= 1,
\end{align}
(quasi-reciprocity) for each $i=1,\dots,n$,
there is some $y_i\in \bbP$ such that
\begin{align}
\label{eq:p3}
\frac{p_{i,s}}
{p_{i,0}}
\frac{p_{i,d_i}}
{p_{i,d_i-s}}=y_i^{2s},
\qquad s=1,\dots, d_i.
\end{align}

Conversely, suppose that a family $\bfp$ in \eqref{eq:p1a}
satisfies  properties \eqref{eq:p2} and \eqref{eq:p3}.
First we note that such a $y_i$ is unique, since 
any semifield $\bbP$ is torsion-free \cite[Section 5]{Fomin02}.
Next we define $z_{i,s}\in \bbP$ ($i=1,\dots,n; s=0,\dots,d_i$) by
\begin{align}
\label{eq:u1}
\frac{p_{i,s}}
{p_{i,0}}=y_i^s z_{i,s}.
\end{align}
In particular, we have $z_{i,0}=1$.
Then, substituting \eqref{eq:u1} in \eqref{eq:p3}, we obtain
\begin{align}
\label{eq:u2}
z_{i,s}z_{i,d_i}z_{i,d_i-s}^{-1}=1,
\qquad
s=1,\dots, d_i.
\end{align}
In particular, by setting $s=d_i$, we have $z_{i,d_i}^2=1$.
Once again, since $\bbP$ is torsion-free,
we have $z_{i,d_i}=1$.
Then, again by \eqref{eq:u2}, we have the reciprocity 
$z_{i,s}=z_{i,d_i-s}$ ($s=1,\dots,d_i-1$).
Meanwhile, by \eqref{eq:p2} and \eqref{eq:u1}, we have
\begin{align}
p_{i,0}=
\frac{1}{\bigoplus_{s=0}^{d_i}
z_{i,s} {y}_i^{ s}}.
\end{align}
Then, by \eqref{eq:u1} again, we recover the specialization
\eqref{eq:pu1}.
Finally, it is straightforward to recover the mutation \eqref{eq:ymut1}
from \eqref{eq:pmut1} and \eqref{eq:p3}.
Furthermore, by  \eqref{eq:u1}, 
one can also confirm that the  coefficients $z_{i,s}$ 
do not mutate.

\subsection{Generalized cluster algebras and Laurent property}
\label{subsec:generalized}

Let $\bbT_n$ be the $n$-regular tree whose edges are labeled by the numbers
$1,\dots,n$.
Following \cite{Fomin02},
let us write $t {\buildrel k \over -} t'$
 if the vertices $t$ and $t'$ of $\bbT_n$ are connected by the edge labeled by $k$.

\begin{defn}
A {\em $(\bfd,\bfz)$-cluster pattern $\Sigma$ in $\bbP$\/} is an assignment of a seed
$\Sigma_t$ in $\bbP$ to each vertex $t$ of $\bbT$ such that, if 
$t {\buildrel k \over -} t'$ then the assigned seeds $\Sigma_t$ and $\Sigma_{t'}$
are obtained from each other by the $(\bfd,\bfz)$-mutation at $k$.
\end{defn}

We fix a vertex $t_0$ of $\bbT_n$ and call it the {\em initial vertex}.
Accordingly, the assigned seed $\Sigma_{t_0}=(\bfx_{t_0},\bfy_{t_0},B_{t_0})$
at $t_0$ is called the
{\em initial seed}. Let us write, for simplicity,
\begin{align}
\bfx_{t_0}=\bfx=(x_1,\dots,x_n),
\quad
\bfy_{t_0}=\bfy=(y_1,\dots,y_n),
\quad
B_{t_0}=B=(b_{ij})_{i,j=1}^n.
\end{align}
On the other hand, for the seed $\Sigma_{t}=(\bfx_{t},\bfy_{t},B_{t})$ assigned to a general
vertex $t$ of $\bbT_n$, we write
\begin{align}
\bfx_{t}=(x^t_1,\dots,x^t_n),
\quad
\bfy_{t}=(y^t_1,\dots,y^t_n),
\quad
B_{t}=(b^t_{ij})_{i,j=1}^n.
\end{align}

\begin{defn} The {\em generalized cluster algebra $\calA$ associated with
a   $(\bfd,\bfz)$-cluster pattern $\Sigma$ in $\bbP$\/} is a
$\bbZ\bbP$-subalgebra of $\calF$ generated by all $x$-variables
$x^t_i$ ($t\in \bbT, i=1,\dots,n$) occurring in $\Sigma$.
It is denoted by $\calA=\calA(\bfx,\bfy,B;\bfd,\bfz)$,
where $(\bfx,\bfy,B)$ is the initial seed of $\Sigma$.
\end{defn}

%

For any $(\bfd,\bfz)$-cluster pattern in $\bbP$,
each $x$-variable $x^t_i$ is expressed as a subtraction-free rational function
of $\bfx$ with coefficients in $\bbQ\bbP$.
The following stronger property due to \cite{Chekhov11} is of fundamental importance.
\begin{thm}[{Laurent property \cite[Theorem 2.5]{Chekhov11}}]
\label{thm:Laurent1}
For any $(\bfd,\bfz)$-cluster pattern in $\bbP$,
each $x$-variable $x^t_i$ is expressed as a Laurent polynomial
of $\bfx$ with coefficients in $\bbZ\bbP$.
\end{thm}

\subsection{Example}
\label{subsec:ex1}

As the simplest nontrivial example,
we consider
$\bfd=(2,1)$, $\bfz=(z_{1,1})$,
and an initial seed $(\bfx,\bfy,B)$ in $\bbP$
such that
\begin{align}
B=
\begin{pmatrix}
0 & -1 \\
1 & 0\\
\end{pmatrix}.
\end{align}
(This example also appears in the proof of \cite[Theorem 2.7]{Chekhov11}.)
Accordingly,
\begin{align}
\hat{y}_1=y_1x_2,
\quad
\hat{y}_2=y_2x_1^{-1}.
\end{align}
We note that
\begin{align}
DB=
\begin{pmatrix}
0 & -2 \\
1 & 0\\
\end{pmatrix},
\quad
BD=
\begin{pmatrix}
0 & -1 \\
2 & 0\\
\end{pmatrix},
\end{align}
which are the initial exchange matrices for
 ordinary cluster algebras of type $B_2=C_2$.
Set $\Sigma(1)=(\bfx(1),\bfy(1),B(1))$ to be the initial seed $(\bfx,\bfy,B)$,
and consider the seeds
$\Sigma(t)=(\bfx(t), \bfy(t),B(t))$ ($t=2,\dots,7$)
obtained by the following sequence of alternative mutations
of  $\mu_1$ and  $\mu_2$.
\begin{align}
\label{eq:seedmutseq3}
&\Sigma(1)
\
\mathop{\leftrightarrow}^{\mu_{1}}
\
\Sigma(2)
\
\mathop{\leftrightarrow}^{\mu_{2}}
\
\Sigma(3)
\
\mathop{\leftrightarrow}^{\mu_{1}}
\
\Sigma(4)
\
\mathop{\leftrightarrow}^{\mu_{2}}
\
\Sigma(5)
\
\mathop{\leftrightarrow}^{\mu_{1}}
\
\Sigma(6)
\
\mathop{\leftrightarrow}^{\mu_{2}}
\
\Sigma(7).
\end{align}
By \eqref{eq:bmut1}, we have
\begin{align}
 B(t)=(-1)^{t+1}B.
\end{align}
Then, using the exchange relations \eqref{eq:ymut1} and \eqref{eq:xmut1},
we obtain the explicit expressions of $x$- and $y$-variables
in Table \ref{tab:data1},
where we set $z_{1,1}=z$ for simplicity.
We observe the same periodicity of mutations of seeds for
the ordinary cluster algebras of type $B_2=C_2$.

\begin{table}
\begin{alignat*}{3}
&
\begin{cases}
x_1(1)=x_1\\
x_2(1)=x_2,\\
\end{cases}
&&
\hskip-70pt
\begin{cases}
y_1(1)=y_1\\
y_2(1)=y_2,\\
\end{cases}
\\
\allowbreak
&
\begin{cases}
\displaystyle
x_1(2)=x_1^{-1}\frac{1+z\hat{y}_1+\hat{y}_1^2}{1\oplus z y_1\oplus y_1^2}\\
x_2(2)=x_2,\\
\end{cases}
&&
\hskip-70pt
\begin{cases}
y_1(2)=y_1^{-1}\\
y_2(2)=y_2(1\oplus z y_1\oplus y_1^2),\\
\end{cases}
\notag
\\
&
\begin{cases}
\displaystyle
x_1(3)=x_1^{-1}\frac{1+z\hat{y}_1+\hat{y}_1^2}{1\oplus z y_1\oplus y_1^2}\\
\displaystyle
x_2(3)=x_2^{-1}
\frac{1 + \hat{y}_2 +z \hat{y}_1\hat{y}_2+
\hat{y}_1^2 \hat{y}_2}
{1\oplus y_2\oplus z y_1y_2
\oplus y_1^2y_2},\\
\end{cases}
&&
\hskip-70pt
\begin{cases}
y_1(3)=y_1^{-1}(1\oplus y_2 \oplus z y_1y_2\oplus y_1^2 y_2)\\
y_2(3)=y_2^{-1}(1\oplus z y_1\oplus y_1^2)^{-1},\\
\end{cases}
\notag
\\
\allowbreak
&
\begin{cases}
\displaystyle
x_1(4)=x_1x_2^{-2}
\frac{
1+ 2 \hat{y}_2+ \hat{y}_2^2
 + z\hat{y}_1\hat{y}_2+ z\hat{y}_1\hat{y}_2^2 + \hat{y}_1^2\hat{y}_2^2}
{1\oplus 2 y_2\oplus y_2^2
 \oplus zy_1y_2\oplus zy_1y_2^2 \oplus y_1^2y_2^2}
\\
\displaystyle
x_2(4)=x_2^{-1}
\frac{1 + \hat{y}_2 +z \hat{y}_1\hat{y}_2+
\hat{y}_1^2 \hat{y}_2}
{1\oplus y_2\oplus z y_1y_2
\oplus y_1^2y_2},\\
\end{cases}
\notag
\\
\allowbreak
&
&&
\hskip-70pt
\begin{cases}
y_1(4)=y_1(1\oplus y_2 \oplus z y_1y_2\oplus y_1^2 y_2)^{-1}\\
y_2(4)=y_1^{-2}y_2^{-1}(1\oplus 2 y_2\oplus y_2^2\\
\qquad \qquad \oplus zy_1y_2\oplus zy_1y_2^2 \oplus y_1^2y_2^2
),\\
\end{cases}
\notag
\\
&
\begin{cases}
\displaystyle
x_1(5)=x_1x_2^{-2}
\frac{
1+ 2 \hat{y}_2+ \hat{y}_2^2
 + z\hat{y}_1\hat{y}_2+ z\hat{y}_1\hat{y}_2^2 + \hat{y}_1^2\hat{y}_2^2}
{1\oplus 2 y_2\oplus y_2^2
 \oplus zy_1y_2\oplus zy_1y_2^2 \oplus y_1^2y_2^2}
\\
\displaystyle
x_2(5)=
x_1x_2^{-1}
\frac{1 + \hat{y}_2 }{1\oplus y_2},\\
\end{cases}
\notag
\\
\allowbreak
&
&&
\hskip-70pt
\begin{cases}
y_1(5)=y_1^{-1}y_2^{-1}(1\oplus y_2)\\
y_2(5)=y_1^{2}y_2(1\oplus 2 y_2\oplus y_2^2\\
\qquad \qquad \oplus zy_1y_2\oplus zy_1y_2^2 \oplus y_1^2y_2^2
)^{-1},\\
\end{cases}
\notag
\\
\allowbreak
&
\begin{cases}
\displaystyle
x_1(6)=x_1\\
\displaystyle
x_2(6)=x_1x_2^{-1}\frac{1+\hat{y}_2}{1\oplus y_2},\\
\end{cases}
&&
\hskip-70pt
\begin{cases}
\displaystyle
y_1(6)=y_1y_2(1\oplus y_2)^{-1}\\
y_2(6)=y_2^{-1},\\
\end{cases}
\notag
\\
\allowbreak
&
\begin{cases}
x_1(7)=x_1\\
x_2(7)=x_2,\\
\end{cases}
&&
\hskip-70pt
\begin{cases}
y_1(7)=y_1\\
y_2(7)=y_2.\\
\end{cases}
\notag
\end{alignat*}
\caption{$x$- and $y$-variables  for sequence \eqref{eq:seedmutseq3}.}
\label{tab:data1}
\end{table}

\section{Structure of seeds in generalized cluster patterns}
The goal of this section is to establish some basic structural results on seeds
in a $(\bfd,\bfz)$-cluster pattern
which are parallel to the ones in \cite{Fomin07}.

\subsection{$X$-functions and $Y$-functions}
Let us temporarily regard $\bfy=(y_i)_{i=1}^n$, and $\bfz
=(z_{i,s})_{i=1,\dots,n;s=1,\dots,d_i-1}$ with 
$z_{i,s}=z_{i,d_i-s}$ as formal variables.
Let $\bbQ_{\mathrm{sf}}(\bfy, \bfz)$  be the {\em universal
semifield\/} of $\bfy$ and $\bfz$,
which consists of the rational functions in $\bfy$ and  $\bfz$
with subtraction-free expressions \cite{Fomin07}.
Let $\mathrm{Trop}(\bfy, \bfz)$  be the {\em tropical
semifield\/} of $\bfy$ and $\bfz$,
which is the multiplicative abelian group freely generated by  $\bfy$ and  $\bfz$
with  {\em tropical sum\/} $\oplus$ defined by
\begin{align}
\label{eq:tsum1}
\left(
\prod_i y_i^{a_i} \prod_{i,s} z_{i,s}^{a_{i,s}}
\right)
\oplus
\left(
\prod_i y_i^{b_i} \prod_{i,s} z_{i,s}^{b_{i,s}}
\right)
=
\prod_i y_i^{\min(a_i,b_i)} \prod_{i,s} z_{i,s}^{\min(a_{i,s},b_{i,s})}.
\end{align}

\begin{defn}
A $(\bfd,\bfz)$-cluster pattern with {\em principal coefficients\/}
is a
  $(\bfd,\bfz)$-cluster pattern in $\bbP=\mathrm{Trop}(\bfy, \bfz)$
with initial seed $(\bfx,\bfy,B)$,
where $\bfx$ and $B$ are arbitrary.
\end{defn}

\begin{defn}
Let $\Sigma$ be the $(\bfd,\bfz)$-cluster pattern with principal coefficients
and initial seed $(\bfx,\bfy,B)$.
By the Laurent property in Theorem \ref{thm:Laurent1},
each $x$-variable $x^t_i$ in $\Sigma$ is expressed as
 $X^t_i(\bfx,\bfy,\bfz)\in \bbZ\bbP[\bfx^{\pm1}]$
with $\bbP=\mathrm{Trop}(\bfy, \bfz)$.
We call them the {\em $X$-functions} of $\Sigma$.
\end{defn}

For principal coefficients,
actually we have the following stronger result than Theorem \ref{thm:Laurent1},
which is parallel to
\cite[Proposition 11.2]{Fomin03a} and
\cite[Proposition 3.6]{Fomin07}.

\begin{prop}
\label{prop:poly1}
 We have
\begin{align}
\label{eq:X1}
X^t_i(\bfx,\bfy,\bfz)\in \mathbb{Z}[\bfx^{\pm1},\bfy,\bfz].
\end{align}
\end{prop}
\begin{proof}
We follow the argument in the proof of \cite[Proposition 11.2]{Fomin03a}.
Let $p$ be any variable in $\bfy$ or $\bfz$.
Let us view $X^t_i(\bfx,\bfy,\bfz)$ as a Laurent polynomial in $p$, say $h(p)$,
with coefficients of Laurent polynomials in the rest of the variables in
$\bfx$, $\bfy$, and $\bfz$.
We show that $h(p)$ is a polynomial in $p$ with nonzero constant term
having subtraction-free rational expression
 by the induction on the distance between $t$ and  $t_0$ in $\bbT_n$.
The crucial point is that the coefficients 
$p_{k,s}=z_{k,s}y_k^s/
\bigoplus_{r=0}^{d_k} z_{k,r}y_k^r$ in the mutation \eqref{eq:xmut1} are normalized
as \eqref{eq:p2}. Since $\bbP=\mathrm{Trop}(\bfy, \bfz)$,
this means that 
$p_{k,s}$ ($s=0,\dots, d_r$)
are polynomials in $p$,
and there  is no common factor in $p$.
Thus, the right hand side of \eqref{eq:xmut1} is a polynomial in $p$
with nonzero constant term
having subtraction-free rational expression
by the induction hypothesis and the ``trivial lemma" in
\cite[Lemma 5.2]{Fomin03a}.
\end{proof}

\begin{defn}
Let  $\Sigma$ be  the
$(\bfd,\bfz)$-cluster pattern  in the universal semifield $\bbQ_{\mathrm{sf}}(\bfy,\bfz)$
with  initial seed $(\bfx,\bfy,B)$.
Each $y$-variable $y^t_i$  in $\Sigma$ is expressed as a subtraction-free rational function
$Y^t_i(\bfy,\bfz)\in \bbQ_{\mathrm{sf}}(\bfy, \bfz)$.
We call them the {\em $Y$-functions} of $\Sigma$.
\end{defn}

Due to the universal property of the semifield $\bbQ_{\mathrm{sf}}(\bfy,\bfz)$
\cite[Definition 2.1]{Fomin07},
the following fact holds.
\begin{lem}
\label{lem:yeval1}
For any $(\bfd,\bfz)$-cluster pattern in $\bbP$ with the same initial exchange matrix $B$ as above, we have
\begin{align}
y^t_{i} = Y^t_i\vert_{\bbP}(\bfy,\bfz),
\end{align}
where the right hand side stands for the evaluation of $Y^t_i(\bfy,\bfz)$
in $\bbP$.
\end{lem}

\subsection{$c$-vectors, $F$-polynomials, and $g$-vectors}
Let us extend the notions of 
$c$-vectors, $F$-polynomials, and $g$-vectors in \cite{Fomin07} to 
a $(\bfd,\bfz)$-cluster pattern with principal coefficients.

\subsubsection{$C$-matrices and $c$-vectors}
For a $(\bfd,\bfz)$-cluster pattern with principal coefficients,
each $y$-variable $y^t_i\in \mathrm{Trop}(\bfy, \bfz)$ is, by definition, a Laurent monomial of $\bfy$ and $\bfz$
with coefficient 1.
The following simple fact was observed in \cite{Iwaki14b} in the special case.
\begin{lem}
\label{lem:yev1}
Each $y$-variable $y^t_i$ is actually a Laurent monomial of $\bfy$
with coefficient 1.
\end{lem}
\begin{proof}
This is equivalent to saying that the frozen coefficients $\bfz$ never enter in $y^t_i$.
This is true for the initial $y$-variables.
Then, the claim can be shown by  induction on the distance  between $t$ and  $t_0$ in $\bbT_n$,
 by inspecting
the mutation \eqref{eq:ymut1} and the definition of the tropical sum \eqref{eq:tsum1}.
\end{proof}

\begin{defn}
Let $\Sigma$ be a $(\bfd,\bfz)$-cluster pattern with principal coefficients.
Let us express each $y$-variable $y^t_j$  in $\Sigma$ as
\begin{align}
\label{eq:yc1}
y^t_j =Y^t_i\vert_{\mathrm{Trop}(\bfy, \bfz)}(\bfy,\bfz)=
 \prod_{i=1}^n y_i^{c_{ij}^t}.
\end{align}
The resulting matrices $C^t=(c^t_{ij})_{i,j=1}^n$ 
and their column vectors $c^t_{j}=(c^t_{ij})_{i=1}^n$
are called
the {\em $C$-matrices\/} and
the {\em $c$-vectors\/}  of $\Sigma$,
respectively.
\end{defn}

The following mutation/recurrence formula provides a combinatorial
description of $c$-vectors.

\begin{prop} The $c$-vectors
of a $(\bfd,\bfz)$-cluster pattern with principal coefficients
satisfy the following recurrence relation:
\par
(initial condition)
\begin{align}
 c_{ij}^{t_0} = \delta_{ij},
\end{align}
(recurrence relation)  for $t {\buildrel k \over -} t'$,
\begin{align}
\label{eq:cmut1}
c^{t'}_{ij}=
\begin{cases}
-c^{t}_{ik} & j=k\\
c^{t}_{ij} +c^t_{ik}[\ve d_kb^t_{kj}]_+
+[-\ve c^t_{ik}]_+ d_kb^t_{kj}
&j\neq k,
\end{cases}
\end{align}
where $\ve=\pm1$ and it is independent of the choice of the sign $\ve$.
\end{prop}
\begin{proof}
As already remarked in the proof of Lemma \ref{lem:yev1}, 
for a $(\bfd,\bfz)$-cluster pattern with principal coefficients,
the mutation \eqref{eq:ymut1} is  simplified as
\begin{align}
\label{eq:ymut2}
{y}^{t'}_i&=
\begin{cases}
\displaystyle
{y}^{t}_k{}^{-1}
 & i=k\\
 \displaystyle
{y}^t_i
\Biggl(
{y}^t_k{}^{[\ve {b}^t_{ki}]_+} 
\Biggr)^{d_k}
\Biggl(
\bigoplus_{s=0}^{d_k}
 {y}^t_k{}^{\ve s}
\Biggr)^{-{b}^t_{ki}}
& i\neq k.\\
\end{cases}
\end{align}
This is equivalent to \eqref{eq:cmut1} due to the following formula
in $\mathrm{Trop}(\bfy, \bfz)$:
\begin{align}
\label{eq:trop1}
\frac{1
}
{\displaystyle
\bigoplus_{s=0}^{d_k}
\Biggl(\prod_{j=1}^n 
y_j^{\ve c^t_{jk}}
\Biggr)^{ s}
}
=
\Biggl(\prod_{j=1}^n 
y_j^{[-\ve c^t_{jk}]_+}
\Biggr)^{d_k}
.
\end{align}
\end{proof}

We observe that
the above relation  coincides with the one for the $c$-vectors
of the ordinary cluster pattern with principal coefficients
and  initial seed $(\bfx,\bfy,DB)$ in
\cite[Proposition 5.8]{Fomin02}.
Therefore, we have the following result.

\begin{prop}
\label{prop:c1}
 The $c$-vectors
of the $(\bfd,\bfz)$-cluster pattern with principal coefficients
and  initial seed $(\bfx,\bfy,B)$
coincide with
 the $c$-vectors
of the ordinary cluster pattern with principal coefficients and
  initial seed $(\bfx,\bfy,DB)$.
\end{prop}

Alternatively, one can relate these $c$-vectors with
 the $c$-vectors
of the ordinary cluster pattern with principal coefficients and
  initial seed $(\bfx,\bfy,BD)$ as follows.
Let us introduce
  \begin{align}
  \label{eq:ct1}
  \tilde{c}^t_{ij} = d^{-1}_i c^t_{ij} d_j.
  \end{align}
  Then, $ \tilde{c}_{ij}^{t_0} = \delta_{ij}$, and \eqref{eq:cmut1} is rewritten as
\begin{align}
\label{eq:ctmut1}
\tilde{c}^{t'}_{ij}=
\begin{cases}
-\tilde{c}^{t}_{ik} & j=k\\
\tilde{c}^{t}_{ij} +\tilde{c}^t_{ik}[\ve b^t_{kj}d_j]_+
+[-\ve \tilde{c}^t_{ik}]_+ b^t_{kj}d_j
&j\neq k.
\end{cases}
\end{align}

Therefore, we have the following result.

\begin{prop}
\label{prop:ct1}
 The   $\tilde{c}$-vectors, which are the column vectors in \eqref{eq:ct1},
of the $(\bfd,\bfz)$-cluster pattern with principal coefficients
and  initial seed $(\bfx,\bfy,B)$
coincide with
 the $c$-vectors
of the ordinary cluster pattern with principal coefficients and
  initial seed $(\bfx,\bfy,BD)$.
\end{prop}

We need this alternative description for the description of the $g$-vectors below.

\subsubsection{$F$-polynomials}
Thanks to Proposition \ref{prop:poly1},
the following definition makes sense.
\begin{defn}
Let $\Sigma$ be a $(\bfd,\bfz)$-cluster pattern with principal coefficients.
For each $t\in \bbT_n$ and $i=1,\dots,n$,
a polynomial $F^t_i(\bfy,\bfz)\in \mathbb{Z}[\bfy,\bfz]$
is defined by the specialization of the $X$-function $X^t_i(\bfx,\bfy,\bfz)$
of $\Sigma$ with
$x_1=\cdots = x_n =1$. 
They are called the {\em $F$-polynomials\/} of $\Sigma$.
\end{defn}

%

The following mutation/recurrence formula provides a combinatorial
description of $F$-polynomials.

\begin{prop}[{cf.~\cite[Proposition 5.1]{Fomin07}}]  The $F$-polynomials
for a $(\bfd,\bfz)$-cluster pattern with principal coefficients
satisfy the following recurrence relation:
\par
{(initial condition)}
\begin{align}
 F_i^{t_0}=1,
\end{align}
(recurrence relation)  for $t {\buildrel k \over -} t'$,
\begin{align}
\label{eq:Fmut1}
F^{t'}_{i}=
\begin{cases}
\displaystyle
F^t_k{}^{-1}
\Biggl(\prod_{j=1}^n 
y_j^{[-\ve c^t_{jk}]_+}
F^{t}_j{}^{[-\ve {b}^t_{jk}]_+}
\Biggr)^{d_k}
\displaystyle
\sum_{s=0}^{d_k}
z_{k,s} 
\Biggl(\prod_{j=1}^n 
y_j^{\ve c^t_{jk}}
F^{t}_j{}^{\ve{b}^t_{jk}}\Biggr)^{ s}
& i=k\\
F^t_i&i\neq k,
\end{cases}
\end{align}
where $\ve=\pm1$ and it is independent of the choice of the sign $\ve$.
\end{prop}
\begin{proof}
By specializing the mutation \eqref{eq:xmut1} with $\bbP=\mathrm{Trop}(\bfy, \bfz)$,
we obtain
\begin{align}
\label{eq:Xmut1}
X^{t'}_i&=
\begin{cases}
\displaystyle
X^t_k{}^{-1}
\Biggl(\prod_{j=1}^n 
X^{t}_j{}^{[-\ve {b}^t_{jk}]_+}
\Biggr)^{d_k}
\frac{\displaystyle
\sum_{s=0}^{d_k}
z_{k,s} 
\Biggl(\prod_{j=1}^n 
y_j^{\ve c^t_{jk}}
X^{t}_j{}^{\ve {b}^t_{jk}}\Biggr)^{ s}
}
{\displaystyle
\bigoplus_{s=0}^{d_k}
\Biggl(\prod_{j=1}^n 
y_j^{\ve c^t_{jk}}
\Biggr)^{ s}
}
 & i=k\\
X^t_i& i\neq k.\\
\end{cases}
\end{align}
Then, specializing it with $x_1=\dots x_n=1$,
and using \eqref{eq:trop1},
we obtain \eqref{eq:Fmut1}.
\end{proof}

\subsubsection{$G$-matrices and $g$-vectors}
Let $\Sigma$ be the $(\bfd,\bfz)$-cluster pattern with principal coefficients
and initial seed $(\bfx,\bfy,B)$.
Let $\mathbb{Z}[\bfx^{\pm1},\bfy,\bfz]$ be the one
in
Proposition \ref{prop:poly1}.
Following \cite{Fomin07}, we introduce a $\bbZ^n$-grading in
$\mathbb{Z}[\bfx^{\pm1},\bfy,\bfz]$ as follows:
\begin{align}
\deg(x_i) = \bfe_i,
\quad
\deg(y_i) = -{\bfb}_j,
\quad
\deg(z_{i,r})=0.
\end{align}
Here, $\bfe_i$ is the $i$th unit vector  of
$\bbZ^n$, and ${\bfb}_j=\sum_{i=1}^n {b}_{ij} \bfe_i$
is the $j$th column of the initial matrix $B=(b_{ij})_{i,j=1}^n$.
Note that $\deg(\hat{y}_i)=0$  by \eqref{eq:yhat1}.

\begin{prop}[{cf. \cite[Proposition 6.1]{Fomin07}}]
\label{prop:gvec1}
The $X$-functions  are homogeneous with respect to the 
$\bbZ^n$-grading.
\end{prop}
\begin{proof}
We repeat the same argument in \cite[Proposition 6.1]{Fomin07}
using the induction on the distance  between $t$ and  $t_0$ in $\bbT_n$. 
Using  \eqref{eq:xmut1} and Lemma \ref{lem:yeval1}
specialized to a $(\bfd,\bfz)$-cluster pattern with principal coefficients, we have
\begin{align}
\label{eq:Xmut2}
X^{t'}_i&=
\begin{cases}
\displaystyle
X^t_k{}^{-1}
\Biggl(\prod_{j=1}^n 
X^{t}_j{}^{[-\ve {b}^t_{jk}]_+}
\Biggr)^{d_k}
\frac{\displaystyle
\sum_{s=0}^{d_k}
z_{k,s} Y^t_k{}^{\ve s}\vert_{\calF}
(\hat{\bfy},\bfz)
}
{\displaystyle
\bigoplus_{s=0}^{d_k}
z_{k,s} {Y}^t_k{}^{\ve s}\vert_{\mathrm{Trop}(\bfy,\bfz)}
(\bfy,\bfz)
}
 & i=k\\
X^t_i& i\neq k.\\
\end{cases}
\end{align}
Then, the right hand side is homogeneous due to the induction hypothesis.
\end{proof}

\begin{defn}
Let $\Sigma$ be the $(\bfd,\bfz)$-cluster pattern with principal coefficients
and initial matrix $(\bfx,\bfy,B)$.
Thanks to Proposition \ref{prop:gvec1},
 the degree vector $\deg(X^t_i)$ of each $X$-function $X^t_i$
 of $\Sigma$ is defined.
 Let us express it as
\begin{align}
\deg(X^t_j)=\sum_{i=1}^n g^t_{ij}\bfe_i.
\end{align}
The resulting matrices $G^t=(g^t_{ij})_{i,j=1}^n$ 
and their column vectors $g^t_{j}=(g^t_{ij})_{i=1}^n$
are called
the {\em $G$-matrices\/} and
the {\em $g$-vectors\/} of $\Sigma$,
respectively.
\end{defn}

The following mutation/recurrence formula provides a combinatorial
description of $g$-vectors.

\begin{prop} The $g$-vectors
of the $(\bfd,\bfz)$-cluster pattern with principal coefficients
and initial seed $(\bfx,\bfy,B)$
satisfy the following recurrence relation:
\par
(initial condition)
\begin{align}
 g_{ij}^{t_0} = \delta_{ij},
\end{align}
(recurrence relation)  for $t {\buildrel k \over -} t'$,
\begin{align}
\label{eq:gmut1}
g^{t'}_{ij}=
\begin{cases}
\displaystyle
-g^{t}_{ik}
+\sum_{\ell=1}^n g^t_{i\ell} [-\ve {b}^t_{\ell k}d_k]_+ 
-\sum_{\ell=1}^n {b}_{i\ell}[-\ve c^t_{\ell k}d_k]_+
& j=k\\
g^{t}_{ij}
&j\neq k,
\end{cases}
\end{align}
where $\ve=\pm1$ and it is independent of the choice of the sign $\ve$.
\end{prop}
\begin{proof}
This is obtained by comparing the degrees of
both  sides of \eqref{eq:Xmut1}.
\end{proof}

By using  the $\tilde{c}$-vectors in
\eqref{eq:ct1},
the relation \eqref{eq:gmut1} is rewritten as follows.
\begin{align}
\label{eq:gmut2}
g^{t'}_{ij}=
\begin{cases}
\displaystyle
-g^{t}_{ik}
+\sum_{\ell=1}^n g^t_{i\ell} [-\ve {b}^t_{\ell k}d_k]_+ 
-\sum_{\ell=1}^n {b}_{i\ell}d_{\ell}[-\ve \tilde{c}^t_{\ell k}]_+
& j=k\\
g^{t}_{ij}
&j\neq k.
\end{cases}
\end{align}

Having Proposition \ref{prop:ct1} in mind,
we observe that this relation
 coincides with the one for the $g$-vectors
of the ordinary cluster pattern with principal coefficients
and  initial seed $(\bfx,\bfy,BD)$ in
\cite[Proposition 6.6]{Fomin07}.
Therefore, we have the following result.

\begin{prop}
\label{prop:g1}
 The $g$-vectors
of the $(\bfd,\bfz)$-cluster pattern with principal coefficients
and  initial seed $(\bfx,\bfy,B)$
coincide with the $g$-vectors
of the ordinary cluster pattern with principal coefficients and
  initial seed $(\bfx,\bfy,BD)$.
\end{prop}

For the sake of completeness, let us also present the counterpart of
Proposition \ref{prop:ct1}.
Let us introduce
\begin{align}
\label{eq:gt1}
\tilde{g}^t_{ij}=  d_i g^t_{ij}d_j^{-1}.
\end{align}
Then, the relation \eqref{eq:gmut1} is also rewritten as
\begin{align}
\label{eq:gmut3}
\tilde{g}^{t'}_{ij}=
\begin{cases}
\displaystyle
-\tilde{g}^{t}_{ik}
+\sum_{\ell=1}^n \tilde{g}^t_{i\ell} [-\ve d_{\ell}b^t_{\ell k}]_+ 
-\sum_{\ell=1}^n d_i b_{i\ell}[-\ve c^t_{\ell k}]_+
& j=k\\
\tilde{g}^{t}_{ij}
&j\neq k.
\end{cases}
\end{align}
Having Proposition \ref{prop:c1} in mind,
we observe that this relation
coincides with the one for the $g$-vectors 
of the ordinary cluster pattern with principal coefficients
and  initial seed $(\bfx,\bfy,DB)$.
Therefore, we have the following result.

\begin{prop}
\label{prop:gt1}
 The $\tilde{g}$-vectors,
which are the column vectors in \eqref{eq:gt1},
of the $(\bfd,\bfz)$-cluster pattern with principal coefficients
and  initial seed $(\bfx,\bfy,B)$
coincide with the $g$-vectors
of the ordinary cluster pattern with principal coefficients and
  initial seed $(\bfx,\bfy,DB)$.
\end{prop}

We see a duality between the $c$-vectors and the $g$-vectors
in Propositions \ref{prop:c1}, \ref{prop:ct1}, \ref{prop:g1}, and \ref{prop:gt1}.
In particular, the $c$-vectors are  associated with the matrix $DB$,
while the $g$-vectors are associated with the matrix $BD$.
This is somewhat suggested from the beginning in the monomial parts
in the relations \eqref{eq:ymut1} and \eqref{eq:xmut1}.

\subsubsection{Sign-coherence}

\begin{defn}
Let $\Sigma$ be a $(\bfd,\bfz)$-cluster pattern with principal coefficients.
A $c$-vector $c^t_j$ of $\Sigma$ is said to be {\em sign-coherent\/}
if it is nonzero and all components are either nonnegative or nonpositive.
\end{defn}

The following proposition is parallel to \cite[Proposition 5.6]{Fomin07}.

\begin{prop}[{cf.~\cite[Proposition 5.6]{Fomin07}}]
For any $(\bfd,\bfz)$-cluster pattern with principal coefficients,
the following two conditions are equivalent.
\begin{itemize}
\item[(i)] Any $F$-polynomial $F^t_i(\bfy,\bfz)$ has constant term 1.
\item[(ii)] Any $c$-vector $c^t_i$ is sign-coherent.
\end{itemize}
\end{prop}
\begin{proof}
This is proved by a parallel argument to the one for
\cite[Proposition 5.6]{Fomin07} by using
the recursion relation \eqref{eq:Fmut1} for the $F$-polynomials.
We omit the detail.
\end{proof}

In the ordinary case it was conjectured in \cite[Conjecture 5.6]{Fomin07}
that the sign-coherence holds for any $c$-vector of any 
cluster pattern with principal coefficients.
This was proved by \cite[Theorem 1.7]{Derksen10}
when the initial exchange matrix $B$ is skew-symmetric,
and very recently it was proved in full generality by \cite[Corollary 5.5]{Gross14}.
Since our $c$-vectors are identified with the $c$-vectors of
some ordinary cluster pattern with principal coefficients
by Proposition \ref{prop:c1},
we obtain the following theorem as a corollary of
\cite[Corollary 5.5]{Gross14}.

\begin{thm}
Any $c$-vector of any
$(\bfd,\bfz)$-cluster pattern with principal coefficients is sign-coherent.
\end{thm}

As a consequence of the sign-coherence, 
we also obtain the following duality between the $C$- and $G$-matrices
by applying \cite[Eq.~(3.11)]{Nakanishi11a} (see also \cite[Proposition 3.2]{Nakanishi11c}),
which is valid under the sign-coherence property.
Recall that for a skew-symmetrizable matrix $B$
 the matrix $DB$ is still skew-symmetrizable.

\begin{prop}[{cf.~\cite[Eq.~(3.11)]{Nakanishi11a}}]
Let $C^t$ and $G^t$ be the $C$- and $G$-matrices at $t\in \bbT_n$
of
any
$(\bfd,\bfz)$-cluster pattern $\Sigma$ with principal coefficients.
Let $R=(r_i\delta_{ij})_{i,j=1}^n$ be a diagonal matrix with positive diagonal entries
such that $RDB$ is skew-symmetric.
Then, the following relation holds.
\begin{align}
R^{-1}D^{-1} (G^t)^T DRC^t = I.
\end{align}
\end{prop}
\begin{proof}
This is obtained by combining \cite[Eq.~(3.11)]{Nakanishi11a}
with Propositions \ref{prop:c1} and \ref{prop:gt1}.
\end{proof}


\subsection{Main formulas}
Finally, we present the main formulas,
 which express the $x$- and $y$-variables
of any $(\bfd,\bfz)$-cluster pattern  $\Sigma$ in any semifield $\bbP$
in terms of $F$-polynomials, $c$-vectors, and $g$-vectors
defined for
 the same initial exchange matrix of $\Sigma$.

\begin{thm}[{cf.~\cite[Proposition 3.13]{Fomin07}}]
For any $(\bfd,\bfz)$-cluster pattern in $\bbP$,
the following formula holds.
\begin{align}
\label{eq:ycf1}
 y^t_i= \prod_{j=1}^n
y_j^{c^t_{ji}}
\prod_{j=1}^n F^t_j\vert_{\bbP}(\bfy,\bfz)^{{b}^t_{ji}}.
\end{align}
\end{thm}
\begin{proof} The derivation is parallel to \cite[Proposition 3.13]{Fomin07}.
We apply Lemma \ref{lem:yeval1} to a $(\bfd,\bfz)$-cluster pattern with principal coefficients, and we obtain
\begin{align}
 \hat{y}^t_i = Y^t_i(\hat{\bfy},\bfz).
\end{align}
On the other hand, specializing  \eqref{eq:yhat1} to the
same $(\bfd,\bfz)$-cluster pattern with principal coefficients, we
have
\begin{align}
\begin{split}
 \hat{y}^t_i &= Y^t_i \vert_{\mathrm{Trop}(\bfy,\bfz)}(\bfy,\bfz)\prod_{j=1}^n X^t_j(\bfx,\bfy,\bfz)^{{b}^t_{ji}} 
\\
 &=\prod_{j=1}^n y_j^{c^t_{ji}}\prod_{j=1}^n X^t_j(\bfx,\bfy,\bfz)^{{b}^t_{ji}} 
,
\end{split}
\end{align}
where we used \eqref{eq:yc1} in the second equality.
Thus, we have
\begin{align}
 Y^t_i(\hat{\bfy},\bfz)=\prod_{j=1}^n y_j^{c^t_{ji}}
\prod_{j=1}^n X^t_j(\bfx,\bfy,\bfz)^{{b}^t_{ji}}.
\end{align}
Now, we set $x_1=\dots=x_n=1$. Then, $\hat{\bfy}=\bfy$, and
we obtain
\begin{align}
\label{eq:ycf2}
 Y^t_i (\bfy,\bfz)= \prod_{j=1}^n
y_j^{c^t_{ji}}
\prod_{j=1}^n F^t_j(\bfy,\bfz)^{{b}^t_{ji}}.
\end{align}
Finally, evaluating it in $\bbP$, we obtain \eqref{eq:ycf1}.
\end{proof}

\begin{thm}[{cf.~\cite[Corollary 6.3]{Fomin07}}]
 For any $(\bfd,\bfz)$-cluster pattern in $\bbP$, the following formula holds.
 \begin{align}
 \label{eq:xf1}
x^t_i  & =
\Biggl(\prod_{j=1}^n
x_j^{g^t_{ji}}\Biggr)
\frac{
F^t_i\vert_{\calF}(\hat{\bfy},\bfz)
}
{
F^t_i\vert_{\bbP}({\bfy},\bfz)
}.
\end{align}
\end{thm}

\begin{proof}
The derivation is  parallel to \cite[Corollary 6.3]{Fomin07}.
First, we obtain the following equality
 exactly in
 the same way as 
\cite[Theorem 3.7]{Fomin07}, and we skip its derivation:
\begin{align}
 \label{eq:xf2}
x^t_i &=
\frac{
X^t_i\vert_{\calF}(\bfx,{\bfy},\bfz)
}
{
F^t_i\vert_{\bbP}({\bfy},\bfz)
}.
\end{align}
On the other hand, by the definition of the $g$-vectors, we have
\begin{align}
X^{t}_i (\dots,\gamma_ix_i,\dots;
\dots, \prod_{j=1}^n \gamma_k ^{-{b}_{ki}}y_i, \dots;
\dots, z_{i,r},\dots)
= \Biggl(\prod_{j=1} \gamma_j^{g^{t}_{ji}}\Biggr)
X^t_i(\bfx,\bfy,\bfz).
\end{align}
By setting $\gamma_i = x_i^{-1}$, we have
\begin{align}
F^{t}_i (\hat{\bfy},\bfz)
= \Biggl(\prod_{j=1} x_j^{-g^{t}_{ji}}\Biggr)
X^t_i(\bfx,\bfy,\bfz).
\end{align}
Combining it with \eqref{eq:xf2}, we obtain \eqref{eq:xf1}.
\end{proof}

\subsection{Example} Let us consider the example in
Section \ref{subsec:ex1} again.
From the data in Table \ref{tab:data1}, one can read off the following data
for the $C$-matrix $C(t)$, the $G$-matrix $G(t)$
and the $F$-polynomials $F_i(t)$ for the seed $\Sigma(t)$ with
principal coefficients therein.
\begin{alignat}{3}
C(1)&=\begin{pmatrix}
1 & 0 \\
0 & 1\\
\end{pmatrix},
\quad
&
G(1)&=\begin{pmatrix}
1 & 0 \\
0 & 1\\
\end{pmatrix},
\quad
&&
\begin{cases}
F_1(1)=1\\
F_2(1) = 1,
\end{cases}
\\
C(2)&=\begin{pmatrix}
-1 & 0 \\
0 & 1\\
\end{pmatrix},
\quad
&
G(2)&=\begin{pmatrix}
-1 & 0 \\
0 & 1\\
\end{pmatrix},
\quad
&&
\begin{cases}
F_1(2)=1+ zy_1 + y_1^2\\
F_2(2) = 1,
\end{cases}
\notag
\\
C(3)&=\begin{pmatrix}
-1 & 0 \\
0 & -1\\
\end{pmatrix},
\quad
&
G(3)&=\begin{pmatrix}
-1 & 0 \\
0 & -1\\
\end{pmatrix},
\quad
&&
\begin{cases}
F_1(3)=1+ zy_1 + y_1^2\\
F_2(3) = 1+y_2+zy_1y_2+y_1^2y_2,
\end{cases}
\notag
\\
C(4)&=\begin{pmatrix}
1 & -2 \\
0 & -1\\
\end{pmatrix},
\quad
&
G(4)&=\begin{pmatrix}
1 & 0 \\
-2 & -1\\
\end{pmatrix},
\quad
&&
\begin{cases}
F_1(4)=1+ 2y_2 + y_2^2\\
\qquad\qquad+zy_1y_2+zy_1y_2^2+y_1^2y_2^2\\
F_2(4) = 1+y_2+zy_1y_2+y_1^2y_2,
\end{cases}
\notag
\\
C(5)&=\begin{pmatrix}
-1 & 2 \\
-1 & 1\\
\end{pmatrix},
\quad
&
G(5)&=\begin{pmatrix}
1 & 1 \\
-2 & -1\\
\end{pmatrix},
\quad
&&
\begin{cases}
F_1(5)=1+ 2y_2 + y_2^2\\
\qquad\qquad+zy_1y_2+zy_1y_2^2+y_1^2y_2^2\\
F_2(5) = 1+y_2,
\end{cases}
\notag
\\
C(6)&=\begin{pmatrix}
1 & 0 \\
1 & -1\\
\end{pmatrix},
\quad
&
G(6)&=\begin{pmatrix}
1 & 1 \\
0 & -1\\
\end{pmatrix},
\quad
&&
\begin{cases}
F_1(6)=1\\
F_2(6) = 1+y_2,
\end{cases}
\notag
\\
C(7)&=\begin{pmatrix}
1 & 0 \\
0 & 1\\
\end{pmatrix},
\quad
&
G(7)&=\begin{pmatrix}
1 & 0 \\
0 & 1\\
\end{pmatrix},
\quad
&&
\begin{cases}
F_1(7)=1\\
F_2(7) = 1.
\end{cases}
\notag
\end{alignat}

%
%
%
%
%
%


\bibliography{../../biblist/biblist.bib}
\end{document}